\documentclass[11pt,reqno,a4paper]{amsart}
\usepackage[a4paper,left=35mm,right=35mm,top=30mm,bottom=30mm,marginpar=25mm]{geometry}

\usepackage{graphicx,color}
\usepackage{amsmath,amsfonts,amssymb,amsthm,mathrsfs,amsopn}
\usepackage{latexsym}

\usepackage{mathtools}
\usepackage{microtype}

\usepackage{hyperref}

\usepackage[usenames,dvipsnames,x11names,table]{xcolor}

\allowdisplaybreaks

\numberwithin{equation}{section}

\def\H{\mathcal H}

\def\R{\mathbb R}
\def\N{\mathbb N}

\def\e{\varepsilon}

\def\Div{{\rm div}\,}

\def\Id{{\rm Id}}

\def\spt{{\rm spt}}

\def\00{{\bf 0}}

\newcommand{\tr}{\mbox{tr }}

\renewcommand{\Div}{{\rm div \,}}

\newcommand{\cc}{\subset\subset}

\def\Im{{\rm Im}}
\def\Ker{{\rm Ker}}

\newcommand\res{\mathop{\hbox{\vrule height 7pt width .3pt depth 0pt \vrule height .3pt width 5pt depth 0pt}}\nolimits}

\def\FF{\mathbf{F}}

\def\Id{\mathrm{Id}}
\def\Im{\mathrm{Im}}

\newcommand{\weakto} {\rightharpoonup}                 
\newcommand{\weakstarto}{\stackrel {*} {\weakto}}      

\newtheorem{theorem}{Theorem} [section]

\newtheorem{lemma}[theorem]{Lemma}

\theoremstyle{definition}
\newtheorem{definition}[theorem]{Definition}

\newtheorem*{ack}{Acknowledgements}
\newtheorem*{question}{Question}

\title[Rectifiability of varifolds with locally bounded first variation]{Rectifiability of varifolds with locally bounded first variation with  respect to anisotropic surface energies}

\author[G. De Philippis]{Guido De Philippis}
\address{G.D.P.: SISSA, Via Bonomea 265, 34136 Trieste, Italy}
\email{guido.dephilippis@sissa.it}

\author[A. De Rosa]{Antonio De Rosa}
\address{A.D.R.: Institut f\"ur Mathematik, Universit\"at Z\"urich, Winterthurerstrasse 190, CH-8057 Z\"urich, Switzerland}
\email{antonio.derosa@math.uzh.ch}

\author[F. Ghiraldin]{Francesco Ghiraldin}
\address{F.G.: Max Planck Institut for Mathematics in the Sciences, Inselstrasse 22, 04103 Leipzig, Germany}
\email{Francesco.Ghiraldin@mis.mpg.de}

\begin{document}

\begin{abstract}
We extend Allard's celebrated rectifiability theorem  to the setting of  varifolds  with locally bounded  first variation with respect to an anisotropic integrand. In particular, we identify a sufficient and necessary  condition on the integrand to obtain the rectifiability of every \(d\)-dimensional varifold with locally bounded first variation and positive \(d\)-dimensional density. 
In codimension one, this condition is shown to be equivalent to  the strict convexity  of the integrand with respect to the tangent plane.

\end{abstract}

\maketitle

\section{Introduction}

Allard's rectifiability Theorem,  \cite{Allard}, asserts that every \(d\)-varifold in $\R^n$ with locally bounded  (isotropic) first variation is  \(d\)-rectifiable when restricted to the set of points in $\R^n$ with positive lower \(d\)-dimensional density. It is a natural  question whether this result holds for varifolds whose first variation with respect to an anisotropic integrand is locally bounded.

 More specifically,  for an open set $\Omega\subset\R^n$ and a positive \(C^1\) integrand 
 
 \[
 F:  \Omega\times G(n,d)\to \R_{>0}:=(0,+\infty),
  \]
  where $G(n,d)$ denotes the Grassmannian of $d$-planes in $\R^{n}$, we define  the \emph{anisotropic energy} of a $d$-varifold \(V\in\mathbb V_d(\Omega)\) as 
 \begin{equation}\label{eq:eVintro}
\FF(V,\Omega) := \int_{\Omega\times G(n,d)} F(x,T)\, dV(x,T).
\end{equation}
We also define its  {\em anisotropic first variation} as the order one distribution whose action on \(g\in C_c^1(\Omega,\R^n)\) is given by 
 \begin{equation*}\label{eq:firstvarintro}
 \begin{split}
 \delta_F V(g): =&\,\frac{d}{dt}\FF\big ( \varphi_t^{\#}V\big)\Big|_{t=0}
 \\
 =&\int_{\Omega\times G(n,d)}  \Big[\langle d_xF(x,T),g(x)\rangle+  B_F(x,T):Dg(x)  \Big] dV(x,T),
\end{split}
 \end{equation*}
 where  \(\varphi_t(x)=x+tg(x)\), \( \varphi_t ^{\#}V\) is the image varifold of \(V\) through \(\varphi_t\), 
 \(B_F(x,T)\in \R^n\otimes\R^n\) is an explicitly computable \(n\times n\) matrix 
  and  \(\langle A,B\rangle :=\tr  A^* B\) for \(A,B\in \R^n\otimes \R^n\), see Section~\ref{s:notation} below for the precise definitions and the relevant computations. 
  We  have then the following:
 
\begin{question} 
Is it true that for every $V\in\mathbb V_d(\Omega)$ such that \(\delta_F V\) is a Radon measure in \(\Omega\), the associated varifold $V_*$ defined as\footnote{Here \(\Theta^d_*(x,V)\) is the lower $d$-dimensional density of $V$ at the point $x$, see Section~\ref{s:notation}.}
\begin{equation}\label{v*intro}
V_*:=V\res \{x\in \Omega : \Theta^d_*(x,V)>0\}\times G(n,d)
\end{equation}
is $d$-rectifiable?
\end{question}
We will show    that this is true if (and only if in the case of autonomous integrands) \(F\)  satisfies the following \emph{atomic condition} at every point \(x\in \Omega\).
\begin{definition}\label{atomica}
For a given integrand  $F\in C^1(\Omega\times G(n,d)) $, \(x\in \Omega\)  and a Borel  probability measure $\mu \in \mathcal P (G(n,d))$, let us define  
\begin{equation}\label{eq:A}
A_x(\mu):=\int_{G(n,d)} B_F(x,T)d\mu(T)\in \R^n\otimes \R^n.
\end{equation}
We say that \(F\) verifies the \emph{atomic condition} $(AC)$ at \(x\) if the following two conditions are satisfied:
\begin{itemize}
\item[(i)]  \(\dim\Ker A_x(\mu)\le n-d\) for all \(\mu \in \mathcal P (G(n,d))\),
\item[(ii)]  if  \(\dim\Ker A_x(\mu)= n-d\), then  \(\mu=\delta_{T_0}\) for some \(T_0\in G(n,d)\).
\end{itemize}
\end{definition}
The following Theorem is the main result of this paper, see again Section~\ref{s:notation} for the relevant (standard) definitions:
\begin{theorem}\label{thm:main}
 Let $F\in C^1(\Omega\times G(n,d),\R_{>0}) $ be a positive integrand and let us define 
 \begin{equation}\label{eq:V}
 \mathscr V_F(\Omega)=\\
 \Big\{ V\in\mathbb V_d(\Omega):  \mbox{\(\delta_F V\) is a Radon measure}\Big\}.
  \end{equation}
 Then we have the following:
 \begin{itemize}
 \item[(i)] If \(F\) satisfies the atomic condition at every \(x\in \Omega\), then for every \(V\in  \mathscr V_F(\Omega)\) the associated varifold $V_*$ defined in \eqref{v*intro} is \(d\)-rectifiable.
 \item[(ii)] Assume that \(F\) is autonomous, i.e. that \(F(x,T)\equiv F(T)\): then every $V_*$ associated to a varifold \(V\in  \mathscr V_F(\Omega)\) is \(d\)-rectifiable if and only if \(F\) satisfies the atomic condition.
 \end{itemize}
 \end{theorem}
 
For  the area integrand, \(F(x,T)\equiv 1\), it is easy to verify  that \(B_{F}(x,T)=T\) where we are identifying  \(T\in G(n,d)\)  with the matrix \(T\in (\R^n\otimes\R^n)_{\rm sym}\)  representing the  orthogonal projection onto \(T\), see  Section~\ref{s:notation}. Since \(T\) is  positive semidefinite (i.e. \(T\ge 0\)),  it is easy to check that the \((AC)\) condition is satisfied. In particular  Theorem~\ref{thm:main} provides a new independent proof of Allard's rectifiability theorem.

 \medskip
Since the  atomic condition \((AC)\) is essentially necessary to the validity  of the  rectifiability Theorem~\ref{thm:main}, it is relevant to relate  it to the previous known notions of \emph{ellipticity} (or \emph{convexity}) of \(F\) with respect to the ``plane'' variable \(T\). This task seems to be quite hard in the general case. For  \(d=(n-1)\) we can however completely characterize the integrands satisfying \((AC)\).  Referring again to Sections~\ref{s:notation} and~\ref{sec:ac} 
 for a more detailed discussion, we  recall here 
that  in this case the integrand \(F\) can be equivalently thought as a positive one-homogeneous even function \(G: \Omega\times \mathbb R^n \to \R_{\ge 0}\) via the identification
\begin{equation}\label{eq:G}
G(x, \lambda  \nu):=|\lambda| F(x,  \nu^{\perp}) \qquad\textrm{for all \(\lambda \in \R\) and  \(\nu \in \mathbb S^{n-1}\)}.
\end{equation}
The atomic condition then turns out to be equivalent to the {\em strict  convexity} of \(G\), more precisely:

\begin{theorem}\label{thm:main2}
An integrand \(F: C^1(\Omega\times G(n,n-1),\R_{>0})\) satisfies the atomic condition at \(x\) if and only if the function \(G(x,\cdot)\) defined in~\eqref{eq:G}  is strictly convex.  
\end{theorem}

As we said, we have not been able to obtain   a simple characterization  in the general  situation
 when \(2\le d\le (n-2)\) (while for \(d=1\) the reader can easily verify that an analogous version of Theorem~\ref{thm:main2} holds true). 
 In this respect, let us  recall that the study of convexity notions for integrands defined 
 on the \(d\)-dimensional Grassmannian is an active field of research, where several basic questions are still open, see \cite{BuEwSh63,BuEwSh62} and the survey \cite{PaTho04}.

\medskip
Beside its theoretical interest, the above Theorem has some relevant applications in the study of existence of minimizers of geometric variational problems defined on class of rectifiable sets. 
It can be indeed shown that given an \(\FF\)-minimizing sequence of sets, 
 the limit of the varifolds naturally associated to them is \(\FF\)-stationary (i.e. it satisfies \(\delta_F V =0\)) and has density bounded away from zero. 
 Hence, if \(F\) satisfies \((AC)\), this varifold is rectifiable and it can be shown that its support minimizes $\FF$, see the forthcoming papers~\cite{DePDeRGhi3,DeRosa} and also~\cite{DelGhiMag,DePDeRGhi,DeLDeRGhi16} for similar results obtained with different techniques. 

\medskip

We conclude this introduction with a brief sketch of the proof of Theorem~\ref{thm:main}. The original proof of Allard in~\cite{Allard} (see also~\cite{Luckhaus08} for a quantitative improvement under slightly more general assumptions)  for varifolds with locally bounded variations with respect to the area integrand  heavily relies on the monotonicity formula, which is strongly linked  to the isotropy of  the area integrand, \cite{Allardratio}. 
A completely different strategy must hence be used  to prove Theorem~\ref{thm:main}. 

The idea is to use the notion of {\em tangent measure} introduced by Preiss,~\cite{Preiss}, in order to understand the local behavior of a varifold $V$ with locally bounded first variation. Indeed at \(\|V_*\|\) almost every point\footnote{Here \(\|V_*\|\) is the projection on \(\R^n\) of the measure \(V_*\), see Section~\ref{s:notation}.}, condition \((AC)\) is used  to show that every tangent measure is translation invariant along {\em at least} \(d\) (fixed) directions, while the 
positivity of the lower  \(d\)-dimensional density ensures that there exists at least one tangent measure that  is invariant along {\em at most} \(d\) directions.  
The combination of these facts allows to show that the ``Grassmannian part'' of the varifold $V_*$ at $x$ is a Dirac delta \(\delta_{T_x}\) on a fixed plane 
\({T_x}\), see Lemma~\ref{atomic}. A key step is then to show that \(\|V_*\|\ll \H^{d}\): this is achieved by using ideas borrowed from~\cite{Allard84BOOK} and~\cite{DePRin16}. Once this is obtained, a simple rectifiability criterion, based on the results in~\cite{Preiss} and  stated in Lemma~\ref{rect}, allows to show that \(V_*\) is \(d\)-rectifiable.

\begin{ack}
G.D.P. is supported by the MIUR SIR-grant {\it Geometric Variational Problems} (RBSI14RVEZ). A.D.R. is supported by SNF 159403 {\it Regularity questions in geometric measure theory}. The authors would like to thank U. Menne for some very accurate comments on a preliminary version of the paper that allowed  to  improve the main result.
\end{ack}

\section{Notations and preliminary results}\label{s:notation}
We work on an open set \(\Omega\subset \R^{n}\) and we set   $B_r(x)=\{y\in\R^{n}:|x-y|<r\}$, \(B_r=B_r(0)\)  and $B:=B_{1}(0)$. Given a \(d\)-dimensional vector space \(T\), we will denote \(B^d_r(x)=B_r(x)\cap (x+T)\) and  similarly for \(B_r^d\) and \(B^d\).  

 For a matrix  \(A\in \R^n\otimes\R^n\),  \(A^*\) denotes its transpose. Given   \(A,B\in \R^n\otimes\R^n\) we define \( A:B=\tr A^* B=\sum_{ij} A_{ij} B_{ij}\), so that \(|A|^2= A:A\).

\subsection{Measures  and rectifiable sets}
We denote by \(\mathcal M_+(\Omega)\) (respectively \(\mathcal M(\Omega,\R^m)\), \(m\ge1\)) the set of positive (resp. \(\R^m\)-valued) Radon measure on \(\Omega\). Given a Radon measure \(\mu\) we denote by \(\spt \mu\) its support. For a  Borel set \(E\),  \(\mu\res E\) is the  restriction of \(\mu\) to \(E\), i.e. the measure defined by \([\mu\res E](A)=\mu(E\cap A)\). For a \(\R^m\)-valued Radon measure \(\mu\in \mathcal M(\Omega,\R^m)\) we denote by \(|\mu|\in \mathcal M_+(\Omega)\) its total variation and we recall that, for all open sets \(U\),
\[
|\mu|(U)=\sup\Bigg\{ \int\big \langle \varphi(x) ,d\mu(x)\big\rangle\,:\quad \varphi\in C_c^\infty(U,\R^m),\quad \|\varphi\|_\infty \le 1 \Bigg\}.
 \] 
Eventually, we denote by  $\H^d$   the  $d$-dimensional Hausdorff measure and for a \(d\)-dimensional vector space \(T\subset \R^n\) 
we will often identify \(\H^d\res T\) with the \(d\)-dimensional Lebesgue measure \(\mathcal L^d\) on \(T\approx \R^d\).

A set \(K\subset \R^n\) is said to be \(d\)-rectifiable if it can be covered, 
up to an \(\H^d\)-negligible set, by countably many \(C^1\) $d$-dimensional 
submanifolds.  In the following we will only consider \(\H^d\)-measurable sets. Given a $d$-rectifiable set $K$, we denote $T_xK$ the approximate tangent space of $K$ at $x$, which exists for $\H^d$-almost every point $x \in K$, \cite[Chapter 3]{SimonLN}.  A positive  Radon measure $\mu\in \mathcal M_+(\Omega)$ is said to be  $d$-rectifiable if there exists a $d$-rectifiable set $K\subset \Omega$ such that $\mu= \theta\H^d \res K$ for some Borel function \(\theta: \R^n\to \R_{>0}\).

For  \(\mu\in \mathcal M_+(\Omega) \) we consider its lower  and upper  \(d\)-dimensional densities at \(x\):
\[
\Theta_*^d(x,\mu)=\liminf _{r\to 0} \frac{\mu(B_r(x))}{ \omega_d r^d}, \qquad \Theta^{d*}(x,\mu)=\limsup_{r\to 0} \frac{\mu(B_r(x))}{ \omega_d r^d},
\]
where \(\omega_d=\H^d(B^d)\) is the measure of the \(d\)-dimensional unit ball in \(\R^d\). In case these  two limits are equal, we denote by \(\Theta^d(x,\mu)\) their common value. Note that if $\mu= \theta\H^d \res K$  with \(K\) rectifiable 
then \(\theta(x)=\Theta_*^d(x,\mu)=\Theta^{d*}(x,\mu)\) for \(\mu\)-a.e. \(x\), see~\cite[Chapter 3]{SimonLN}.

If \(\eta :\R^n\to \R^n\) is a Borel map and \(\mu\) is a Radon measure, we let \(\eta_\# \mu=\mu\circ \eta^{-1}\) be the push-forward of \(\mu\) through \(\eta\). Let  $\eta^{x,r}:\R^n\rightarrow \R^n$ be  the dilation map, $\eta^{x,r}(y)= (y-x)/r$. For a positive   Radon measure $\mu\in \mathcal M_+(\Omega) $,  \(x\in \spt \mu\cap \Omega \) and \(r\ll1\), we define 
\begin{equation}\label{blowup}
\mu_{x,r}=\frac{1}{\mu(B_{r}(x))} (\eta^{x,r}_\#\mu)\res B.
\end{equation}
The normalization in~\eqref{blowup} implies, by the Banach-Alaoglu Theorem, that for every sequence $r_i\to 0$ there exists a subsequence $r_{i_j}\to 0$ and a Radon measure $\sigma \in \mathcal M_+ (B) $, called {\em tangent measure} to \(\mu\) at \(x\), such that 
\[
\mu_{x,r_{i_j}} \weakstarto\sigma.
\]
 We collect all tangent measures  to \(\mu\) at \(x\) into \({\rm Tan}(x,\mu)\subset \mathcal M_+ (B)\). 
 The next Lemma shows that \({\rm Tan}(x,\mu)\) is not trivial at \(\mu\)-almost every point  where \(\mu\) has positive lower \(d\)-dimensional density and that furthermore there is always a tangent measure which looks at most \(d\)-dimensional on a prescribed ball (a similar argument can be used to show that \({\rm Tan}(x,\mu)\) is always not trivial at \(\mu\) almost every point without any assumption on the \(d\)-dimensional density,  see~\cite[Corollary 2.43]{AFP}).
 
 \begin{lemma}\label{almeno}
Let $\mu\in\mathcal M_+ (\Omega)$ be a Radon measure. Then for every $x \in \Omega$ such that $\Theta_*^d(x,\mu)>0$ and for every $t\in(0,1)$, there exists a tangent measure $\sigma_t\in {\rm Tan}(x,\mu)$ satisfying
\begin{equation}\label{supre}
\sigma_t(\overline{B_t})\geq t^d.
\end{equation}
\end{lemma}
\begin{proof}

{\em Step 1}: We claim that for every $x\in \Omega$ such that $\Theta_*^d(x,\mu)>0$, it holds
\begin{equation}\label{sup}
\limsup_{r\to0}\frac{\mu(B_{tr}(x))}{\mu(B_{r}(x))}\geq t^d, \qquad \forall t\in (0,1).
\end{equation}
More precisely, we are going to show that 
\[
\big\{ x\in \spt \mu: \textrm{\eqref{sup} fails}\big\}\subset \big\{ x\in \spt \mu: \Theta^d_*(x,\mu )=0\big\},
\]
which  clearly implies that~\eqref{sup} holds for every $x \in \Omega$ with positive lower $d$-dimensional density. Let indeed \(x\in\spt \mu\) be such that \eqref{sup} fails, then there exist $ t_0\in (0,1)$,  $\bar \e>0$, and \(\bar r>0\) such that 
\[
 \mu(B_{t_0 r}(x))\leq(1-\bar \e)t_0^d\,\mu(B_{r}(x)) \qquad\textrm{for all \(r\le \bar r\).}
\] 
Iterating this inequality, we deduce that 
$$\mu (B_{t_0^k \bar  r}(x))\leq(1-\bar \e)^k t_0^{kd}\mu(B_{\bar r}(x))\qquad\textrm{for all \(k\in \N\)}$$
and consequently
$$\Theta^d_*(x,\mu )\le \lim_{k\to \infty} \frac{\mu (B_{ t_0 ^k \bar r}(x))}{\omega_d ( t_0^k \bar r)^d}=0.$$

\noindent

{\em Step 2}: Let now \(x\) be a point satisfying  \eqref{sup}  and let  $t\in(0,1)$:  there exists a sequence $r_j\downarrow 0$  (possibly depending on  $t$ and on \(x\)), such that
$$
t^d \leq  \limsup_{r\to 0}\frac{\mu(B_{tr}(x))}{\mu(B_{r}(x))}=\lim_{j\to\infty}\frac{\mu(B_{tr_j}(x))}{\mu(B_{r_j}(x))} = \lim_{j\to\infty}\mu_{x,r_j}(B_t)
$$
where \(\mu_{x,r_j}\) is defined as in~\eqref{blowup}. Up to extracting a (not relabelled) subsequence, 
$$\mu_{x,r_j}\weakstarto \sigma_t\in {\rm Tan}(x,\mu).$$ 
 By upper semicontinuity:
\[
\sigma_t(\overline{B_t})\ge \limsup_j\mu_{x,r_j}(\overline{B_t})\ge t^d
\]
which is~\eqref{supre}.
\end{proof}

In order to prove Theorem~\ref{thm:main} we need the following rectifiability 
criterion which is essentially~\cite[Theorem~4.5]{MaFra99}, see also~\cite[Theorem 16.7]{Mattila}. For the sake of readability, we postpone its proof to Appendix~\ref{Lemma}.

\begin{lemma}\label{rect}
Let \(\mu\in\mathcal M_+(\Omega)\) be a Radon measure such that the following two properties hold:
\begin{itemize}
\item[(i)] For \(\mu\)-a.e. \(x\in \Omega\), \(0<\Theta^d_*(x,\mu)\le \Theta^{d*}(x,\mu)<+\infty\).
\item[(ii)] For \(\mu\)-a.e. \(x\in \Omega\) there exists \(T_x \in G(n,d)\) such that every \(\sigma\in {\rm Tan}(x,\mu)\)  is translation invariant along \(T_x\), 
i.e.
\[
\int \partial_{e} \varphi \,d\sigma  =0\qquad\mbox{for every \(\varphi\in C_c^1(B)\) and every \(e\in T_x\)}.
\]
\end{itemize}
Then \(\mu\) is \(d\)-rectifiable, i.e. \(\mu=\theta \H^d\res K\) for some \(d\)-rectifiable set \(K\) and  Borel function \(\theta:\R^n\to \R_{>0}\). Furthermore \(T_x K=T_x\) for \(\mu\)-a.e. \(x\).
\end{lemma}

\subsection{Varifolds and integrands} We denote by $G(n,d)$ the Grassmannian of (un-oriented) $d$-dimensional 
linear subspaces in $\R^{n}$ (often referred to as $d$-planes) 
and given any set \(E\subset\R^n\) we denote by \(G(E)=E\times G(n,d)\) the (trivial) Grassmannian bundle over \(E\). 
We will often identify a \(d\)-dimensional plane \(T \in G(n,d)\) with the matrix \(T\in (\R^n\otimes\R^n)_{\rm sym}\)  representing the {\em orthogonal projection} onto \(T\).

A $d$-varifold  on \(\Omega\) is  a positive  Radon measure $V$ on $G(\Omega)$ and we will denote with  \(\mathbb V_d(\Omega)\)  the set of all \(d\)-varifolds on \(\Omega\).

Given a diffeomorphism $\psi \in C^1(\Omega,\R^n)$, we define the push-forward of $V\in\mathbb V_d(\Omega)$ with respect to $\psi$ as the varifold $\psi^\#V\in \mathbb V_d(\psi(\Omega))$ such that
$$\int_{G(\psi(\Omega))}\Phi(x,S)d(\psi^\#V)(x,S)=\int_{G(\Omega)}\Phi(\psi(x),d_x\psi(S))J\psi(x,S) dV(x,S),$$
for every $\Phi\in C^0_c(G(\psi(\Omega)))$. Here $d_x\psi(S)$ is the image of $S$ under the linear map $d_x\psi(x)$ and 
\[
J\psi(x,S):=\sqrt{\det\Big(\big(d_x\psi\big|_S\big)^*\circ d_x\psi\big|_S\Big)}
 \]
 denotes the $d$-Jacobian determinant of the differential $d_x\psi$ restricted to the $d$-plane $S$, 
 see \cite[Chapter 8]{SimonLN}. Note that the push-forward of a varifold \(V\) is {\em not} the same as the push-forward of the  Radon measure \(V\) through a map $\psi$ defined on $G(\Omega)$ (the latter being  denoted with an expressly different notation: \(\psi_\# V\)).

To a varifold \( V\in \mathbb V_d(\Omega) \), we associate the measure  $\|V\|\in \mathcal M_+(\Omega)$ defined by 
\[
\|V\|(A)=V(G(A))\qquad\textrm{for all \(A\subset \Omega\) Borel.}
\]
Hence  $\|V\|=\pi_\# V$, where \(\pi: \Omega\times G(n,d)\to \Omega\) is the projection onto the first factor and the push-forward  is intended in the sense of Radon measures. By the disintegration theorem for measures, see for instance~\cite[Theorem 2.28]{AFP}, we can write
\[
V(dx,dT)= \|V\|(dx)\otimes \mu_x(dT),
\]
where \(\mu_x\in \mathcal P(G(n,d))\) is a  (measurable) family of parametrized  non-negative measures on the Grassmannian such that  \(\mu_x(G(n,d))=1\).

A \(d\)-dimensional varifold $V\in \mathbb V_d(\Omega)$ is said $d$-rectifiable 
if there exist a $d$-rectifiable set $K$ and a  Borel  function $\theta: \R^n\to \R_{>0}$ such that $V=\theta \H^d \res (K\cap \Omega) \otimes \delta_{T_x K}$.

We will  use the notation
\[
\Theta^d_*(x,V)=\Theta^d_*(x,\|V\|)\qquad\textrm{and}\qquad\Theta^{d*}(x,V)=\Theta^{d*}(x,\|V\|)
\] 
for the upper and lower \(d\)-dimensional densities of \(\|V\|\). In case $\Theta^d_*(x,V)=\Theta^{d*}(x,V)$, we denote their common value $\Theta^d(x,V)$.

As already specified in the introduction, we will associate to any $d$-varifold $V$, its ``at most \(d\)-dimensional'' part $V_*$ defined as
\begin{equation*}\label{v*}
V_*:=V\res \{x\in \Omega : \Theta^d_*(x,V)>0\}\times G(n,d).
\end{equation*}
Note that 
\begin{equation*}\label{v*1}
\|V_*\|=\|V\|\res \{x\in \Omega : \Theta^d_*(x,V)>0\}
\end{equation*}
and thus, by the Lebesgue-Besicovitch differentiation Theorem~\cite[Theorem 2.22]{AFP}, for \(\|V_*\|\) almost every point (or equivalently for \(\|V\|\) almost every \(x\) with \( \Theta^d_*(x,V)>0\)) 
\begin{equation}\label{v*2}
\lim_{r\to 0} \frac{\|V_*\|(B_r(x))}{\|V\|(B_r(x))}=1.
\end{equation}
In particular,
\begin{equation}\label{v*3}
\Theta^d_*(x,V_*)>0\qquad\mbox{for \(\|V_*\|\)-a.e. \(x\)}.
\end{equation}

Let  \(\eta^{x,r}(y)=(y-x)/r\), as in \eqref{blowup} we define 
\begin{equation*}\label{blowupvarifolds}
V_{x,r}:=\frac{r^d}{\|V\|(B_{r}(x))}\big((\eta^{x,r})^\#V \big)\res G(B), 
\end{equation*}
where the additional factor $r^d$ is due to the presence of the $d$-Jacobian determinant of the differential $d\eta^{x,r}$ in the definition of push-forward of varifolds. Note that, with the notation of~\eqref{blowup}:
\begin{equation*}\label{blowupmisurevarifold}
\|V_{x,r}\|= \|V\|_{x,r}.
\end{equation*}
The following Lemma is based on a simple Lebesgue point argument, combined with the separability of \(C^0_c(G(\Omega))\), see  for instance \cite[Proposition 9]{DLOW}. 
 
\begin{lemma}\label{conv}
For \(\|V\|\)-almost every point  \(x\in \Omega\)  and every sequence \(r_j\to 0\) there is a subsequence \(r_{j_i}\) such that 
\begin{equation*}\label{e:conv}
V_{x,r_{j_i}}(dy,dT)=\|V\|_{x,r_{j_i}}(dy)\otimes \mu_{x+yr_{j_i}}(dT) \weakstarto \sigma(dy) \otimes \mu_{x}(dT)=:V^\infty(dy,dT)
\end{equation*}
with \(\sigma\in {\rm Tan} (x, \|V\|)\). 
\end{lemma}
We call any varifold \(V^\infty\in \mathbb V_d(B)\) arising as  in Lemma~\ref{conv} a {\em tangent varifold} to \(V\) at \(x\) and we collect them into \({\rm Tan}(x, V)\). The key point of the Lemma above is that the ``Grassmannian" part $\mu_y^\infty$ of a tangent varifold $V^\infty$ equals $\mu_x$ for every $y \in \Omega$: it therefore neither depends on the space variable \(y\), nor  on the chosen blow-up sequence $(r_j)$. 
Informally, we could write:
 \[
``\, {\rm Tan}(x, V) = {\rm Tan}(x, \|V\|) \otimes \mu_x(dT) \,".
 \]
We furthermore note that, as a consequence of~\eqref{v*2},
\begin{equation}\label{v*4}
{\rm Tan}(x, \|V_*\|)={\rm Tan}(x, \|V\|)\quad\mbox{and}\quad{\rm Tan}(x, V_*)={\rm Tan}(x, V)\quad\mbox{ \(\|V_*\|\)-a.e.}
\end{equation}

\medskip
The anisotropic integrand  that we consider is a \(C^1\) function
$$
F: G(\Omega) 
\longrightarrow \R_{>0}.
$$
Since our results are local in nature, up to restricting to a compactly contained open subset of $\Omega$, we can assume  the existence of two positive constants \(\lambda, \Lambda\) such that
\begin{equation*}\label{cost per area}
0 < \lambda \leq F(x,T) \leq \Lambda<\infty\qquad\textrm{for all \((x,T)\in  G(\Omega)\).}
\end{equation*}
Given \(x\in \Omega\), we will also consider the ``frozen'' integrand 
\begin{equation}\label{frozen}
F_x:G(n,d) \to (0,+\infty), \qquad F_x(T):= F(x,T).
\end{equation}
As in \eqref{eq:eVintro}, we define the {\em anisotropic energy} of \(V\in \mathbb V_d(\Omega)\) as 
 \begin{equation*}\label{eq:eV}
\FF(V,\Omega) := \int_{G(\Omega)} F(x,T)\, dV(x,T). 
\end{equation*}

For a vector field \(g\in C_c^1(\Omega,\R^n)\), we consider the family of functions \(\varphi_t(x)=x+tg(x)\), and we note that they are diffeomorphisms of \(\Omega\) into itself for \(t\) small enough. The {\em anisotropic first variation} is  defined as 
\[
\delta_F V(g):=\frac{d}{dt}\FF\big ( \varphi_t^{\#}V,\Omega\big)\Big|_{t=0}.
\]
It can be easily shown, see Appendix~\ref{appendixvariation}, that
 \begin{equation}\label{eq:firstvariation}
\delta_F V(g)
=\int_{G(\Omega)} \Big[\langle d_xF(x,T),g(x)\rangle+ B_F(x,T):Dg(x)  \Big] dV(x,T),
 \end{equation}
where  the matrix \(B_F(x,T)\in \R^n\otimes\R^n\) is uniquely defined by 
\begin{equation}\label{eq:B}
\begin{split}
B_F(x,T): L&:= F(x,T) (T:L)+\big\langle d_T F(x,T) ,\,T^\perp\circ L\circ T +(T^\perp\circ L\circ T)^*\big\rangle\\
&=:F(x,T) (T :L)+ C_F(x,T):L
\qquad \mbox{for all \(L\in \R^n\otimes \R^n\).} 
\end{split}
 \end{equation}  
Note that, via  the identification of a \(d\)-plane \(T\) with the orthogonal projection onto it, \(G(n,d)\) can be thought as a subset of \(\R^n\otimes \R^n\) and this gives the  natural identification:
 \begin{multline*}
 {\rm Tan}_T G(n,d) =\big\{ S \in \R^n\otimes\R^n: S^*=S,\quad T\circ S\circ T=0,\quad T^\perp\circ S\circ T^\perp=0\big\},
 \end{multline*}
 see Appendix~\ref{appendixvariation} for more details. 
 We are going to use the  following properties of  \(B_F(x,T)\)  and \(C_F(x,T)\), 
 which immediately follow from~\eqref{eq:B}:
 \begin{gather} 
 |B_F(x,T)-B_F(x,S)|\le C(d,n, \|F\|_{C^1}) \big( |S-T|+\omega(|S-T|)\big)
  \label{prop:B1}
 \\ 
 C_F(x,T): v\otimes w=0 \quad\textrm{for all \(v,w\in T\)}. \label{prop:B2}
 \end{gather}
 where \(\omega\) is the modulus of continuity of \(T\mapsto d_TF(x,T)\) (i.e.: a concave increasing function with \(\omega(0^+)=0\)). We also note that trivially
\begin{equation}\label{operatore}
|\delta_F V(g)|\leq  \|F\|_{C^1(\spt g)}\|g\|_{C^1}\|V\|(\spt(g)),
\end{equation}
and that, if \(F_x\) is the frozen integrand \eqref{frozen}, then \eqref{eq:firstvariation} reduces to
 \begin{equation*}\label{eq:firstvariationfroze}
\delta_{F_x} V(g)=\int_{G(\Omega)}  B_F(x,T):Dg(y)\, dV(y,T).
 \end{equation*}
 We say that a varifold \(V\) has {\em locally bounded anisotropic first variation} if \(\delta_F V\) is a Radon measure on \(\Omega\), i.e. if 
 \[
 |\delta_F V(g)|\le C(K) \|g\|_{\infty}\quad\textrm{for all \(g\in C^1_c(\Omega,\R^n)\) with \(\spt g \subset K\cc \Omega\).}
 \]
 Furthermore, we will say that \(V\) is {\em \(\FF\)-stationary} if \(\delta_FV =0\).
 
 We conclude this section with the following simple result which shows that every tangent varifold 
 to a varifold having locally bounded anisotropic first variation is \(\FF_x\)-stationary.
  
\begin{lemma}\label{stationary}
Let \(V\in \mathbb V_d(\Omega)\) be a \(d\)-dimensional varifold with locally bounded anisotropic first variation. Then, for \(\|V\|\)-almost every point, every \(W\in {\rm Tan} (x, V)\) is \(\FF_x\)-stationary, i.e. \(\delta_{F_x}W =0\). 
Moreover, if \(W( dy, dT)=\sigma(dy)\otimes \mu_x(dT)\) for some \(\sigma\in {\rm Tan} (x,\|V\|)\) (which by Lemma~\ref{conv} happens \(\|V\|\)-a.e. \(x\)), then 
\begin{equation}\label{talk}
\partial_e \sigma=0 \qquad \textrm{for all  \(e\in T_x:=\Im A_x(\mu_x)^*\)}
 \end{equation}
 in the sense of distributions, where  \(A_x(\mu_x)\) is defined in~\eqref{eq:A}.
\end{lemma}

\begin{proof}Let \(x\) be a point such that the conclusion of Lemma~\ref{conv} holds true and such that
\begin{equation}\label{locallybounded}
\limsup_{r \to 0} \frac{| \delta_F V |(B_{r}(x))}{\|V\|(B_{r}(x))}= C_x <+\infty.
\end{equation}
Note that, by Lemma~\ref{conv} and Lebesgue-Besicovitch differentiation Theorem,~\cite[Theorem 4.7]{SimonLN},  this is the case for \(\|V\|\)-almost every point.
We are going to prove the Lemma at every such a point. 

Let \(r_i\) be a sequence such that \(V_{x,{r_i}}(dy,dT)\weakstarto W(dy,dT)=\sigma(dy)\otimes \mu_{x} (dT)\), \(\sigma\in {\rm Tan}(x, \|V\|)\). For $g\in C_c^1(B,\R^n)$, we define $g_i:=g\circ \eta^{x,r_i} \in C_c^1(B_{r_i}(x),\R^n)$ and we compute
\begin{equation*}
\begin{split}
\delta_{F_x} V_{x,r_i}(g)&=\int_{G(B)} B_F(x,T):Dg(y) \, d V_{x,r_i}(y,T)\\
&=\frac{r_i^d}{\|V\|(B_{r_i}(x))}\int_{G(B_{r_i}(x))} \!\!\!B_F(x,T):Dg\left (\eta^{x,r_i}(z)\right )J\eta^{x,r_i}(z,T) \, d V(z,T)\\
&=\frac{r_i}{\|V\|(B_{r_i}(x))}\int_{G(B_{r_i}(x))} B_F(x,T):Dg_i(z) \, d V(z,T)\\
 &=r_i\frac{\delta_{F_x} V(g_i)}{\|V\|(B_{r_i}(x))} =r_i\frac{\delta_{F} V(g_i)+\delta_{(F_x-F)} V(g_i)}{\|V\|(B_{r_i}(x))}.
\end{split}
\end{equation*}
Combining \eqref{operatore}, \eqref{locallybounded} and since $r_i\|D g_i\|_{C^0}= \|D g\|_{C^0}$, we get
\begin{equation*}\label{variazioneblowup}
\begin{split}
\left |\delta_{F_x} V_{x,r_i}(g)\right| &\leq r_i\frac{|\delta_F V|(B_{r_i}(x)) \|g\|_{\infty}}{\|V\|(B_{r_i}(x))}\\
&\quad+r_i\frac{\|F-F_x\|_{C^1(B_{r_i}(x))}\|g_i\|_{C^1}\|V\|(B_{r_i}(x))}{\|V\|(B_{r_i}(x))}\\
&\leq  r_i C_x \|g\|_\infty +o_{r_i}(1)\|g\|_{C^1}
\rightarrow 0,
\end{split}
\end{equation*}
which implies \(\delta_{F_x}W =0\). Hence, recalling the Definition~\eqref{eq:A} of \(A_x(\mu)\),
 for every $g \in  C_c^1(B,\R^n)$: 
$$
 0= \delta_{F_x} W(g) = \int_{B} A_x(\mu_x) : Dg(y)  \, d\sigma (y).
 $$
 Therefore  \(A_x(\mu_x)D \sigma=0\) in the sense of distributions, which is equivalent to~\eqref{talk}, since \( \Ker A_x(\mu_x)=(\Im A_x(\mu_x)^*)^\perp\).
\end{proof}

\section{Intermediate lemmata}
\noindent To prove the sufficiency part of Theorem~\ref{thm:main}, there are two key steps:
\begin{itemize}
\item[(i)] Show that the ``Grassmannian'' part of the varifold $V_*$ is concentrated on a single plane;
\item[(ii)] Show that \(\|V_*\|\ll \H^d\).
\end{itemize}
In this section we prove these steps, in Lemma~\ref{atomic} and Lemma~\ref{abscont} respectively. 

\begin{lemma}\label{atomic}
Let \(F\) be an integrand satisfying condition \((AC)\) at every \(x\) in \(\Omega\) and let \(V\in \mathscr V_F(\Omega)\), see~\eqref{eq:V}. Then, for $\|V_*\|$-a.e. $x\in \Omega$,  $\mu_x=\delta_{T_0}$  for some \(T_0\in G(n,d)\).
\end{lemma}
\begin{proof}
Let \(t\le t(d)\ll1\) to be fixed later. By  Lemmata \ref{conv}, \ref{stationary} and \ref{almeno} and by~\eqref{v*4}, for $\|V_*\|$-a.e. $x$ there exist a sequence $r_i\to 0$ and a tangent measure $\sigma$ such that 
$$
\|V_*\|_{x,r_i} \weakstarto \sigma, \qquad  
(V_*)_{x,r_{i}} \weakstarto \sigma \otimes \mu_x,  \qquad 
\sigma (\overline{B_t})\geq t^d
$$
and
$$\partial_e \sigma=0 \quad \textrm{for all \(e\in T_x=\Im A_x(\mu_x)^*\)}.$$

Let us  now show that if \(t(d)\) is sufficiently small, then \(\mu_x=\delta_{T_0}\).  
Assume by contradiction that $\mu_x$ is not a Dirac delta: from the \((AC)\) condition of $F$, 
this implies  that $\dim \Ker A_x(\mu_x)^*< n-d$ and consequently that $\dim( T_x )>d$. 
This means that $\sigma$ is invariant by translation along at least  $d+1$ directions 
and therefore there exists $Z\in G(n,d+1)$, a probability measure $\gamma\in \mathcal P(Z^\perp)$ defined in the 
linear space $Z^\perp$ and supported in $B^{n-d-1}_{1/\sqrt{2}}$, and a constant $c\in\R$, such that we can decompose \(\sigma\) in 
 the cylinder $B^{d+1}_{1/\sqrt{2}}\times B^{n-d-1}_{1/\sqrt{2}} \subset Z\times Z^\perp$ as 
 $$
 \sigma\res B^{d+1}_{1/\sqrt{2}}\times B^{n-d-1}_{1/\sqrt{2}}=c\H^{d+1}\res (Z\cap B^{d+1}_{1/\sqrt{2}}) \otimes \gamma,
 $$
  where $c\leq2^{(d+1)/2}\omega_{d+1}^{-1}$ since  $\sigma(B_1)\leq1$. Taking $t(d)<\frac1{2\sqrt{2}}$, the ball $\overline{B_t}$ is contained in the cylinder $B^{d+1}_{1/\sqrt{2}}\times B^{n-d-1}_{1/\sqrt{2}}$ and hence 
$$t^d\leq \sigma(\overline{B_t})\leq \sigma( B_t^{d+1}\times B^{n-d-1}_{1/\sqrt{2}})\leq C(d) t^{d+1}, $$ 
which is a contradiction if $t(d)\ll 1$. 
\end{proof}

The next Lemma is inspired by the ``Strong Constancy Lemma'' of Allard~\cite[Theorem 4]{Allard84BOOK}, see also~\cite{DePRin16}.

\begin{lemma}\label{t:integrality}
Let \(F_j:G(B)\to \R_{>0}\) be a sequence of \(C^1\) integrands and let   
$V_j\in \mathbb V_d(G(B))$ be a sequence of $d$-varifolds equi-compactly supported in $B$
 (i.e.  such that \(\spt \|V_j\|\subset K \subset\subset B\)) with $\|V_j\|(B)\leq 1$.  
If there exist $N>0$ and  $S\in G(n,d)$ such that
\begin{itemize}
\item [(1)] $|\delta_{F_{j}} V_j |(B) +\|F_j\|_{C^1(G(B))}\leq N$,
\item [(2)] \(|B_{F_j}(x,T)-B_{F_j}(x,S)|\le \omega(|S-T|)\) for some modulus of continuity independent on \(j\),
\item [(3)] $\delta_j:=\int_{G(B)} |T-S|  dV_j(z,T) \rightarrow 0$ as \(j\to \infty\),
\end{itemize}
then, up to subsequences, there exists $\gamma \in L^1(B^d, \H^d\res B^d)$ such that for every $0<t<1$ 
\begin{equation}\label{eq:limiteforte}
\Big| (\Pi_S)_{\#}\big(F_j(z, S)\|V_j\|\big) - \gamma\mathcal \H^d\res B^d\Big|(B^d_t ) \longrightarrow 0, 
\end{equation}
where \(\Pi_S:\R^n\to S\) denotes   the orthogonal projection onto \(S\) (which in this Lemma we do not identify with \(S\)).
\end{lemma}

\begin{proof}
To simplify the notation let us simply set \(\Pi=\Pi_S\); we will also denote  with a prime
 the variables in the \(d\)-plane \(S\) so that \(x'=\Pi(x)\). 
 Let \(u_j=\Pi_\# \big(F_j(z, S)\|V_j\|\big) \in\mathcal M_+(B^d)\): then
\[
\langle u_j, \varphi\rangle= \int_{G(B)} \varphi (\Pi(z)) F_j(z,S)d V_j(z,T)\qquad\textrm{for all \(\varphi\in C_c^0(B^d)\).}
\]
Let  \(e\in S\) and, for \(\varphi\in C_c^1(B^d)\), let us denote by \(D'\) the gradient of 
\(\varphi\) with respect to the variables in \(S\), so that \(\Pi^*(D'\varphi)(\Pi(z))=D (\varphi(\Pi(z)))\). We then have in the sense of distributions
\begin{equation}\label{talk0}
\begin{split}
-\langle \partial'_e u_j, \varphi\rangle&=\langle u_j, \partial'_e \varphi \rangle=\int_{G(B)} \langle D' \varphi(\Pi(z)),e\rangle F_j(z,S)\, dV_j(z,T)
\\
&=\int_{G(B)} \langle D' \varphi(\Pi(z)),e\rangle (F_j(z,S)-F_j(z,T)) \,dV_j(z,T)\\
&\quad+\int_{G(B)}  F_j(z,T)(S-T)  :e\otimes \Pi^*(D'\varphi)(\Pi(z))  \, dV_j(z,T)\\
&\quad+\int_{G(B)} \big(C_{F_j}(z,S)-C_{F_j}(z,T)\big):e\otimes\Pi^*(D'\varphi)(\Pi(z)) \, dV_j(z,T)\\
&\quad-\int_{G(B)}  \big\langle d_zF_j(z,T),e\,\varphi(\Pi(z))\big\rangle \,d V_j (z,T)\\
&\quad+\int_{G(B)} \big\langle d_zF_j(z,T),e\,\varphi(\Pi(z))\big\rangle \,dV_j(z,T)\\
&\quad+  \int_{G(B)}\big(F_j(z,T)T+C_{F_j}(z,T)\big):e\otimes D (\varphi(\Pi(z)))\,dV_j(z,T),
\end{split}
\end{equation}
where  we have used that 
\[
\Id:e\otimes \Pi^*(D'\varphi)(\Pi(z)) = S: e\otimes \Pi^*(D'\varphi)(\Pi(z)) =\langle D' \varphi(\Pi(z)),e\rangle
\] 
and \(C_{F_j}(z,S):e\otimes \Pi^*(D'\varphi)(\Pi(z))=0\), since  \(D'\varphi\) and \(e\) belong to \(S\), see~\eqref{prop:B2}.  Let us define the distributions
\[
\begin{split}
\langle X^e_j,\psi \rangle&:=
 \int_{G(B)} \Big((F_j(z, S)-F_j(z,T))\Id+F_j(z,T)(S-T)\\
 &\qquad\qquad+(C_{F_j}(z,S)-C_{F_j}(z,T))\Big):e\otimes \Pi^* \psi(\Pi(z))\,dV_j(z,T)
\end{split}
\]
and
\[
\begin{split}
\langle f^e_j,\varphi \rangle &:=\int_{G(B)} \big\langle d_zF_j(z,T),e\,\varphi(\Pi(z))\big\rangle\, dV_j (z,T),\\
\langle g^e_j,\varphi \rangle&:=-\int_{G(B)}  \Big( \big\langle d_zF_j(z,T),e\,\varphi(\Pi(z))\big\rangle \\
&\qquad\qquad+ \big(F_j(z,T)T+C_{F_j}(z,T)\big):e\otimes \Pi^* D' \varphi(\Pi(z))\Big)\,dV_j(z,T)
\\
&\,=-\delta_{F_j} V_j(e\,\varphi\circ \Pi).
\end{split}
\]
By their very definition, \(X^e_j\) are vector valued Radon measures in \(\mathcal M(B_1^d,\R^d)\) and, by the uniform bound on the \(C^1\) norm of the \(F_j\),~\eqref{prop:B1} and assumptions (2) and (3):
\begin{equation}\label{talk1}
\sup_{|e|=1} |X^e_j |(B^d_1)\to 0\qquad\mbox{as \(j\to \infty\)}.
\end{equation}
Moreover, by the mass bound \(\|V_j\|(B)\le 1\) and assumption (1), \(f_j^e\) and \(g^e_j\) are  also Radon measures satisfying
\begin{equation}\label{talk2}
\sup_j\sup_{|e|=1}|f^e_j|(B^d_1)+|g^e_j|(B^d_1)<+\infty.
\end{equation}
Letting \(e\) vary in an orthonormal base  \(\{e_1,\dots, e_d\}\) of \(S\), we can re-write~\eqref{talk0} as
\begin{equation}\label{talk1/2}
D'  u_j=\Div' X_j+f_j+g_j,
\end{equation}
where \(X_j=(X^{e_1}_j,\dots,X^{e_d}_j)\in \R^d\otimes \R^d\),  \(f_j=(f_j^{e_1},\dots,f_j^{e_d})\) and \(g_j=(g_j^{e_1},\dots,g_j^{e_d})\). 

Let us now choose an arbitrary sequence \(\e_j\downarrow 0\) and a family of smooth approximation of the identity \(\psi_{\e_j}(x')=\e_j^{-d}\psi(x'/\e_j)\), with \(\psi\in C_c^\infty(B_1)\), \(\psi\ge 0\). To prove~\eqref{eq:limiteforte} it is 
enough to show that \(\{v_j:=u_j\star \psi_{\e_j}\}\) is precompact in \(L^1_{\rm loc}(B^d_1)\). 
Note that by convolving~\eqref{talk1/2} we get that \(v_j\) solves
\begin{equation}\label{eq:zero}
D  v_j=\Div Y_j+h_j,
\end{equation}
where, to simplify the notation, we have set \(D=D'\), \(\Div=\Div'\) and 
\[
Y_j:=X_j\star \psi_{\e_j} \in C^\infty_c(B_1^d,\R^n\otimes\R^n), \qquad h_j=(f_j+g_j)\star \psi_{\e_j}\in C^\infty_c(B_1^d,\R^n)
\] 
are smooth functions compactly supported in \(B_1^d\). Note that, by \eqref{talk1}, \eqref{talk2} and the positivity of \(u_j\) 
\[
v_j\geq 0,
\qquad
\int |Y_j|\rightarrow 0
\qquad
\mbox{and} \qquad \sup_j \int |h_j| <+\infty.
\]
We can solve the system \eqref{eq:zero} by taking another divergence  and inverting the Laplacian  using the potential theoretic solution (note that all the functions involved are compactly supported):
\begin{equation}\label{eq:system}
 v_j = \Delta^{-1} \Div (\Div Y_j)   +\Delta^{-1}  \Div h_j.
\end{equation}
Recall that 
\begin{equation}\label{laplacian}
\Delta^{-1} w = E\star w,
\end{equation}
with $E(x) = - c_d |x|^{2-d}$ if $d\geq 3 $ and  $E(x) = c_2 \log|x|$ if $d=2$, for some positive constants $c_d$, depending just on the dimension. 
Hence, denoting by P.V. the principal value,
\[
\begin{split}
\Delta^{-1} \Div (\Div Y_j) (x) &= K\star Y_j (x)\\
&: = {\rm P.V.}\;c_d \int_{\R^d}\frac{(x-y)\otimes(x-y)-|x-y|^2\Id}{|x-y|^{d+2}}:Y_j(y)dy,
\end{split}
\]
and
\[
\Delta^{-1} \Div h_j (x)= G\star h_j(x):= c_d \int_{\R^d}\Big\langle \frac{x-y}{|x-y|^{d}}, h_j(y)  \Big\rangle dy.
\]
By the Frechet-Kolomogorov compactness theorem, the operator 
$h\mapsto G\star h :L_c^1(B_1^d)\rightarrow L^1_{\rm loc}(\R^d)$ is compact (where \(L^1_c(B_1^d)\) are the \(L^1\) functions with compact support in \(B_1^d\)). 
Indeed, for $M\geq 1$, by direct computation one verifies that
\begin{equation}\label{eq:frechet}
\begin{split}
\int_{B^d_M} |G\star h (x + v) - G\star h(x)|dx \leq C |v|\log\left(\frac{eM}{|v|}\right)\int_{B^d_1}|h|dx,
\qquad \forall v\in B^d_{1}.
\end{split}
\end{equation}
In particular,  $\{b_j:=G\star h_j\}$ is precompact  in $L^1_{\rm loc}(\R^d)$. 
The first term is more subtle: the kernel $K$ defines a Calderon-Zygmund operator $Y\mapsto K\star Y$ on Schwarz functions that can be extended to a bounded operator from $L^1$ to $L^{1,\infty}$, \cite[Chapter 4]{GrafakosCla}. In particular we can bound the quasi-norm of $a_j := K\star Y_j$ as
\begin{equation}\label{eq:zyg}
[a_j]_{L^{1,\infty}(\R^d)} := \sup_{\lambda>0} \lambda |\{|a_j|>\lambda \}|  \le C \int_{B^d_1}|Y_j|  \rightarrow 0. 
\end{equation}
Moreover,  $ K\star Y_j\weakstarto 0$ in the sense of distributions, since \(\langle K\star Y_j, \varphi\rangle=\langle Y_j, K\star \varphi\rangle\to 0\) for \(\varphi\in C_c^1(\R^d)\). We can therefore write
\[
0\leq v_j = a_j + b_j,
\]
with $a_j\rightarrow 0$ in $L^{1,\infty}$ by \eqref{eq:zyg}, $a_j\weakstarto  0$ in the sense of distributions and $\{b_j\}$  pre-compact in $L_{\rm loc}^1$ by \eqref{eq:frechet}.
 Lemma \ref{l:balance} below implies that  $v_j$ is strongly precompact in $L_{\rm loc}^1$, which is the desired conclusion. 
\end{proof}

\begin{lemma}\label{l:balance}
Let $\{v_j\}, \{a_j\}, \{b_j\} \subset L^1(\R^d)$ such that
\begin{itemize}
\item[(i)] $0\leq v_j = a_j + b_j $, 
\item[(ii)] $\{b_j\}$ strongly precompact in $L^1_{\rm loc}$,
\item[(iii)] $a_j\rightarrow 0$ in $L^{1,\infty}$ and $a_j \weakstarto 0$ in the sense of distributions.
\end{itemize}
Then  $\{v_j\}$ is strongly precompact in  $L_{\rm loc}^1$.
\end{lemma}
\begin{proof}
It is enough to show that \(\chi |a_j|\to 0\) in \(L^1\) for \(\chi\in C_c^\infty (\R^d)\), \(\chi\ge0\). The first condition implies that $a_j^-\leq |b_j|$, hence the sequence $\{\chi a_j^-\}$ 
is equi-integrable and thus, by (iii) and Vitali convergence Theorem, it converges to zero in \(L_{\rm loc}^1\), hence
\[
\int \chi |a_j|=\int \chi a_j + 2\int \chi a_j^-\to 0,
\]
where the first integral goes to zero by (iii).
\end{proof}

The following Lemma is a key step in the proof of Theorem~\ref{thm:main}:

\begin{lemma}\label{abscont}
Let \(F\) be an integrand satisfying condition \((AC)\) at every \(x\) in \(\Omega\) and let \(V\in \mathscr V_F(\Omega)\), see~\eqref{eq:V}. Then $\|V_*\| \ll \H^d$.
\end{lemma}
\begin{proof}
Since  by \eqref{v*3}, 
$\Theta_*^d(\cdot,V)>0$  $\|V_*\|$-a.e.,  classical differentiation theorems for measures  imply that
\[
\H^d\res  \{\Theta_*^d(\cdot,V)>\lambda\}\le \frac{1}{\lambda}\|V_*\| \qquad \forall \lambda>0,
\]
 see~ \cite[Theorem 6.9]{Mattila}. Hence  \(\H^d\res  \{\Theta_*^d(\cdot,V)>0\}\) is a \(\sigma\)-finite measure and  by the Radon-Nikodym Theorem
\begin{equation}\label{rndecompostion}
\|V_*\| 
= f \H^d\res\{\Theta_*^d(\cdot,V)>0\} + \|V_*\|^s
\end{equation}
for some psitive Borel function  $f$ and  $\|V_*\|^s$ is concentrated on a set $E\subset\{\Theta_*^d(\cdot,V)>0\}$ such that $\H^d(E)=0$: in particular
 $\H^d(\Pi(E))=0$ whenever $\Pi$ is an orthogonal projection onto a $d$-dimensional subspace of $\R^n$. Hence \(\|V_*\|^s\) and  \(f \H^d\res\{\Theta_*^d(\cdot,V)>0\}\)   are  mutually singular Radon measure (the fact that they are Radon measures follows trivially from~\eqref{rndecompostion}).

 We are going to show that \(\|V_*\|^s=0\), which clearly concludes the proof. To this aim, let us assume by contradiction that \(\|V_*\|^s>0\) and let us choose a point \(\bar x\in \Omega\) and a sequence of radii \(r_j\to 0\)  such that:
 \begin{enumerate}
\item[(i)]   
\begin{equation}\label{eq:quattro}
\lim_{j\to \infty} \frac{\|V_*\|^s(B_{r_j}(\bar x))}{\|V_*\|(B_{ r_j}(\bar x))} =\lim_{j\to \infty} \frac{\|V_*\|(B_{r_j}(\bar x))}{\|V\|(B_{ r_j}(\bar x))} =1.
\end{equation}
\item[(ii)] There exists \(\sigma\in {\rm Tan} (\bar x, \|V\|)= {\rm Tan} (\bar x, \|V_*\|)= {\rm Tan} (\bar x, \|V_*\|^s)\), with \(\sigma\res B_{1/2}\ne 0\).
\item[(iii)]
\begin{equation}\label{ow}
\limsup_{j\to\infty} \frac{|\delta_F V|(B_{r_j}(\bar x))}{\|V\|(B_{r_j}(\bar x))}\le C_{\bar x} < + \infty.
\end{equation}
\item[(iv)] 
 \begin{equation}\label{eq:uno}
 V_j := V_{\bar x,r_j}  
 \weakstarto  \sigma \otimes \delta_{S},
 \end{equation}
 where \(S\in G(n,d)\) and \(\partial_e \sigma=0\) for every \(e\in S\).
\end{enumerate}
 Here the first, second and third conditions hold \(\|V_*\|^s\)-a.e. by simple measure theoretic arguments and by~\eqref{v*2} and~\eqref{v*4},  and the fourth one holds \(\|V_*\|^s\)-a.e. as well by combining Lemma~\ref{stationary}, Lemma~\ref{atomic} and~\eqref{v*4}.
  
Fix a smooth cutoff function  $\chi$ with $0\leq \chi\leq 1$, $\spt(\chi)\subset B_1$ and $\chi = 1$ in $B_{1/2}$ and define $W_j:= \chi V_j$ so that 
\[
\|W_j\|=\chi f \H^d\res\{\Theta_*^d(\cdot,V)>0\} + \|W_j\|^s
\]
where \(\|W_j\|^s=\chi\|V_*\|^s\). In particular 
\begin{equation}\label{concentrato}
(\Pi_S)_\#\|W_j\|^s \qquad\text{is concentrated on }\qquad  E_j:=\Pi_S \left(\frac{E-\bar x}{r_j}\right)\cap B^d_1,
\end{equation}
and thus 
\begin{equation}\label{concentrato1}
 \H^d (E_j)=0.
 \end{equation}
  Note furthermore that
\begin{equation}\label{supbound}
\sup_j | \delta_{F_j} W_j|(\R^d)<+\infty,
\end{equation}
where \(F_j(z,T)=F(\bar x+r_j z,T)\). Indeed for $\varphi \in C^\infty_c(B_1,\R^n)$ 
\begin{equation*}
\begin{split}
|\delta_{F_j}W_j(\varphi)| &=| \delta_{F_j}(\chi V_j)(\varphi)|\\
&=\Big| \int  r_j\big\langle d_x F(\bar x+r_j z, T), \chi(z) \varphi(z)\rangle d V_j(z,T)\\
&\quad+ \int B_{F} (\bar x +r_j z,T):D\varphi(z)\, \chi(z) dV_j(z,T)\Big|	\\
&=\Big| \int  r_j \big\langle d_x F(\bar x +r_j z, T), \chi \varphi\rangle  d V_j(z,T)\\
&\quad+ \int B_{F} (\bar x +r_j z,T):D(\chi \varphi)(z) dV_j(z,T)\\
&\quad- \int B_{F} (\bar x +r_j z,T):D \chi(z) \otimes \varphi(z) dV_j(z,T)\Big|  \\
&\leq |\delta_{F_j} V_j(\chi\varphi)|+ \|F\|_{C^1}\|V_j\|(B_1)\|D \chi\|_\infty  \|\varphi\|_\infty     \\
&\leq    r_j\frac{|\delta_F V|(B_{r_j}(\bar x))}{\|V\|(B_{r_j}(\bar x))}\|\varphi\|_\infty+\|F\|_{C^1}\|V_j\|(B_1)\|D \chi\|_\infty  \|\varphi\|_\infty,  
\end{split}
\end{equation*}
so that~\eqref{supbound} follows from~\eqref{ow} and the fact that \(\|V_j\|(B_1)\le 1\). Finally, by~\eqref{eq:uno},
\begin{equation*}\label{eq:conv}
\begin{split}
\lim_{j}\int_{G(B_1)} |T-S| dW_j(z,T)&=\lim_j \int_{G(B_1)}|T-S| \chi(z)  dV_j(z,T)\\
 &= \int_{G(B_1)}|T-S| \chi(z) d\delta_{S}(T)  d\sigma(z)=0.
 \end{split}
\end{equation*}
Hence the sequences of integrands \(\{F_j\}\) and of varifolds  \(\{W_j\}\) satisfy the assumptions of Lemma~\ref{t:integrality} (note indeed that \(B_{F_j}(z,T)=B_{F}(\bar x+r_jz,T)\) so that assumption (2) in Lemma~\ref{t:integrality} is satisfied). Thus we deduce the existence of  $\gamma \in L^1(\H^d\res B^d_1)$ such that, along a (not relabelled) subsequence, for every \(0<t<1\)
\begin{equation}\label{eq:contradiction}
\Big| (\Pi_S)_{\#}(F(\bar x +r_j(\cdot),S)\|W_j\|)- \gamma \mathcal H^d\res B^d\Big| (B^d_t )\longrightarrow 0.
\end{equation}
By~\eqref{eq:quattro} we can substitute  $\|W_j\|^s$ for $\|W_j\|$ in \eqref{eq:contradiction} to get 
\begin{equation*}\label{eq:contradiction2}
\Big| (\Pi_S)_{\#}(F(\bar x +r_j(\cdot),S)\|W_j\|^s)- \gamma \mathcal H^d\res B^d_1\Big| (B^d_t )\longrightarrow 0.
\end{equation*}
By point (ii) above,  \( F(\bar x +r_j(\cdot),S)\|W_j\|^s\weakstarto F(\bar x,S)\chi\sigma\) with  \(\sigma\res B_{1/2}\ne 0\).  Recalling that \(F(\bar x, S)>0\) we then have 
\[
\begin{split}
0&< \big|(\Pi_S)_{\#}(F(\bar x,S)\chi\sigma)\big|(B^d_{1/2}) \\
&\le \liminf _{j\to \infty}  \big|(\Pi_S)_{\#}(F(\bar x +r_j(\cdot),S)\|W_j\|^s)\big|(B^d_{1/2})
\\
&= \liminf _{j\to \infty}  \big|(\Pi_S)_{\#}(F(\bar x +r_j(\cdot),S)\|W_j\|^s)\big|(E_j\cap B^d_{1/2})
\\
&\le  \limsup_{j\to \infty}  \Big|(\Pi_S)_{\#}(F(\bar x +r_j(\cdot),S)\|W_j\|^s)-\gamma \H^d\res B^d \Big| (E_j\cap B^d_{1/2})=0,
\end{split}
\]
since \((\Pi_S)_{\#}\|W_j\|^s\) is concentrated on \(E_j\) and \(\H^d(E_j)=0\), see~\eqref{concentrato} and~\eqref{concentrato1}. This contradiction concludes the proof. 

\end{proof}

\section{Proof of Theorem \ref{thm:main}}
\begin{proof}[Proof of Theorem \ref{thm:main}]
\emph{Step 1: Sufficiency.} 
Let \(F\) be a \(C^1\) integrand satisfying the \((AC)\)  condition at every \(x\in \Omega\) and let \(V\in \mathscr V_{F}(\Omega)\), we want to apply Lemma~\ref{rect} to \(\|V_*\|\). 
Note that, according to Lemma~\ref{abscont} and \eqref{v*3},
\[
\H^d\res\{x\in \Omega: \Theta^d_*(x,\|V_*\|)>0\}\ll \|V_*\|\ll \H^d\res\{x\in \Omega: \Theta^d_*(x,\|V_*\|)>0\}.
\]
Since, by \cite[Theorem 6.9]{Mattila}, \(\H^d\big(\{x\in\Omega: \Theta^{d*}(x,\|V_*\|)=+\infty\}\big)=0 \), we deduce that
\[
0<\Theta^{d}_*(x,\|V_*\|)\le\Theta^{d*}(x,\|V_*\|)<+\infty\qquad\mbox{for \(\|V_*\|\)-a.e. \(x\in \Omega\)},
\]
hence assumption (i) of Lemma~\ref{rect} is satisfied.  By Lemma~\ref{atomic}, \(V_*=\|V_*\|\otimes\delta_{T_x}\) for some \(T_x\in G(n,d)\), and, combining this with Lemma~\ref{stationary} and~\eqref{v*4}, 
 for \(\|V_*\|\)-almost every \(x\in \Omega\) every \(\sigma\in {\rm Tan} (x,\|V_*\|)\) is invariant along the directions of \(T_x\), so that also assumption (ii) of Lemma~\ref{rect} is satisfied. Hence
\[
\|V_*\|=\theta \H^d\res (K\cap \Omega),
\]
for some rectifiable set \(K\) and Borel function \(\theta\). Moreover, again by Lemma~\ref{rect},   \(T_x K=T_x\) for \(\|V_*\|\)-almost every \(x\).  This  proves that \(V_*\) is \(d\)-rectifiable.

\medskip
\noindent
\emph{Step 2: Necessity.}  Let us now assume that $F(x,T)\equiv F(T)$ does not depend on the point, but just on the tangent plane and let us  suppose that $F$ does not verify the atomic condition $(AC)$. We will show the existence of a varifold \(V\in \mathscr V_F(\R^n)\), with positive lower $d$-dimensional density (namely $V=V_*$), which is not $d$-rectifiable.
Indeed the negation of \((AC)\)  means that there exists a probability measure $\mu$ on $G(n,d)$, such that one of the following cases happens:
\begin{itemize}
\item[1)] ${\rm dim}\,\Ker A(\mu) ={\rm dim}\,\Ker A(\mu)^* > n-d$
\item[2)] ${\rm dim}\,\Ker A(\mu) ={\rm dim}\,\Ker A(\mu)^* =n-d$ and  $\mu\neq \delta_{T_0}$,
\end{itemize}
where $A(\mu):=\int_{G(n,d)} B_F(T) \, d\mu(T)$ and $B_F(T)\in\R^n\otimes\R^n$ is constant in $x$.
Let $W:=\Im A(\mu)^*$,  $k={\rm dim}\,W \leq d$ and let us define the varifold 
$$
V(dx,dT):=\H^k \res W(dx) \otimes \mu(dT)\in \mathbb V_{d}(\R^n).
$$ 
Clearly \(V\) is not \(d\)-rectifiable since either \(k<d\) or \(\mu\ne \delta_{W}\).
We start by noticing that $V=V_*$, indeed for  \(x\in W\)
\begin{equation}\label{eq:densitycontr}
\Theta^d(x,V)=\lim_{r \to 0} \frac{\H^k(B_r(x)\cap W)}{\omega_ d r^d}=
\begin{cases} 
1 \qquad &\textrm{if \(k=d\)}\\ 
+\infty &\textrm{if \(k<d\)}.
\end{cases}
\end{equation}
 Let us now prove that \(V\in \mathscr V_F(\R^n)\). For every $g \in  C_c^1(\R^n,\R^n)$, we have
$$
  \delta_F V(g) = \int_{W} A(\mu) : Dg  \, d\H^k  = - \big\langle g, A(\mu) D (\H^k\res W)\big\rangle=0
  $$
since \( D (\H^k\res W)\in W^\perp=[\Im A(\mu)^*]^\perp=\Ker A(\mu)\).  Hence \(V\) is \(\FF\)-stationary and in particular \(V\in \mathscr V_F(\R^n)\) which, together with~\eqref{eq:densitycontr} concludes the proof.
\end{proof}

\section{Proof of Theorem~\ref{thm:main2}}\label{sec:ac}
In this section we prove Theorem~\ref{thm:main2}. As explained in the introduction, it is convenient to identify the Grassmannian \(G(n,n-1)\) with the projective space  \(\mathbb{RP}^{n-1}=\mathbb S^{n-1}\big/\pm\) via the map
\[
\mathbb S^{n-1} \ni\pm \nu\mapsto \nu^\perp.
\]
Hence an \((n-1)\)-varifold \(V\) can be thought as a positive Radon measure \(V\in \mathcal M_{+}(\Omega\times \mathbb S^{n-1})\) even in the \(\mathbb S^{n-1}\) variable, i.e. such that 
\[
V(A\times S)=V(A\times (-S))\qquad\mbox{for all \(A\subset \Omega\), \(S\subset \mathbb S^{n-1}\).}
\]
In the same way, we identify  the integrand \(F:\Omega\times G(n,n-1)\to \R_{>0}\) with a positively one homogeneous even function \(G:\Omega\times \R^{n}\to \R_{\ge0}\) via the equality
\begin{equation}\label{e:identification}
G(x, \lambda  \nu):=|\lambda| F(x,  \nu^{\perp}) \qquad\textrm{for all \(\lambda \in \R\) and  \(\nu \in \mathbb S^{n-1}\)}.
\end{equation}
Note that \(G\in C^1(\Omega\times (\R^{n}\setminus\{0\}))\) and that  by one-homogeneity: 
\begin{equation}\label{eulercod1}
\langle d_{e} G(x, e),e\rangle = G(x,e)\qquad\mbox{for all \(e\in\R^n\setminus \{0\}\).}
\end{equation}
With these identifications, it is a simple calculation to check that:
\begin{equation*}
\begin{split}
\delta_{F} V(g)&=\int_{\Omega\times \mathbb S^{n-1}} \langle d_x G(x,\nu), g(x)\rangle\, dV(x,\nu)\\
& \quad + \int_{\Omega\times \mathbb S^{n-1}}\Big(G(x,\nu)\Id-\nu \otimes d_\nu G(x,\nu)\Big):Dg(x)\, dV(x,\nu),
\end{split}
\end{equation*}
see for instance~\cite[Section 3]{Allard84BOOK} or ~\cite[Lemma A.4]{dephilippismaggi2}. In particular, under the correspondence~\eqref{e:identification}
\begin{equation*}
B_{F}(x,T)=G(x,\nu)\Id-\nu \otimes d_\nu G(x,\nu)=:B_{G}(x,\nu),\qquad T=\nu^\perp.
\end{equation*}
Note that \(B_{G}(x,\nu)=B_{G}(x,-\nu)\) since \(G(x,\nu)\) is even. Hence the atomic condition at \(x\) can be re-phrased as: 
\begin{itemize}
\item[(i)]  \(\dim\Ker A_x(\mu)\le 1\) for all {\em even} probability measures \(\mu \in \mathcal P_{\rm even} (\mathbb S^{n-1})\),
\item[(ii)]  if \(\dim\Ker A_x(\mu)= 1\) then \(\mu=(\delta_{\nu_0}+\delta_{-\nu_0})\big/2\) for some \(\nu_0\in \mathbb S^{n-1}\),
\end{itemize}
where
\[
A_x(\mu)=\int_{\mathbb S^{n-1}} B_G(x,\nu)d\mu(\nu).
\]
We are now ready to prove Theorem~\ref{thm:main2}.
\begin{proof}[Proof of Theorem~\ref{thm:main2}]
Since the \((AC)\) condition deals only with the behavior of the frozen integrand \(G_{x}(\nu)=G(x,\nu)\), for the whole proof \(x\) is fixed and for the sake of readability we drop the dependence on \(x\).

\medskip
\noindent
{\em Step 1: Sufficiency}. Let us assume that \(G:\R^n\to \R\) is even,  one-homogeneous  and strictly convex. We will show that the requirements (i) and (ii)  in the \((AC)\) condition are satisfied. First note that, by one-homogeneity, the strict convexity of \(G\) is equivalent to:
\begin{equation}\label{eq:sconv}
G(\nu)> \langle d_\nu G(\bar \nu),\nu\rangle\qquad \mbox{for all \(\bar \nu,\nu\in \mathbb S^{n-1}\) and  \(\nu\ne \pm \bar \nu\).}
\end{equation}
Plugging \(-\nu\) in~\eqref{eq:sconv} and exploiting the fact that \(G\) is even we obtain
\begin{equation}\label{eq:sconv2}
G(\nu)>| \langle d_\nu G(\bar \nu),\nu\rangle|\qquad \mbox{for all \(\bar \nu,\nu\in \mathbb S^{n-1}\) and  \(\nu\ne \pm \bar \nu\).}
\end{equation}
Let now \(\mu \in \mathcal P_{\rm even} (\mathbb S^{n-1})\) be an even probability measure,
\[
A(\mu)=\int_{\mathbb S^{n-1}} \Big(G(\nu)\Id-\nu\otimes d_{\nu} G(\nu)\Big) d\mu(\nu) 
\]
and assume there exists \(\bar \nu\in \Ker A(\mu)\cap \mathbb S^{n-1}\). We then have
\[
\begin{split}
0&=\langle d_\nu G(\bar \nu),A(\mu)\bar\nu \rangle\\
&=\int_{\mathbb S^{n-1}} \Big\{(G(\bar \nu)G(\nu)-\langle d_\nu G(\bar \nu),\nu\rangle\langle d_\nu G( \nu),\bar \nu\rangle\Big\} d\mu(\nu)\\
&\ge \int_{\mathbb S^{n-1} }\Big\{G(\bar \nu)G(\nu)-\big|\langle d_\nu G(\bar \nu),\nu\rangle\big|\big|\langle d_\nu G( \nu),\bar \nu\rangle\big| \Big\}d\mu(\nu)
\end{split}
\]
where we have used~\eqref{eulercod1}. Inequality~\eqref{eq:sconv2}  implies however that the integrand in the last line of the above equation is strictly positive, unless \(\nu=\pm\bar \nu\) for all \(\nu\in \spt \mu\), which immediately implies that the \((AC)\) condition is satisfied.

\medskip
\noindent
{\em Step 2: Necessity}. Let us assume that \(G\) (or equivalently \(F\)) satisfies the \((AC)\) condition, let \(\nu,\bar \nu \in\mathbb S^{n-1}\), \(\nu\ne \pm\bar \nu\)  and define 
\[
\mu=\frac{1}{4} \big(\delta_{\nu}+\delta_{-\nu}+\delta_{\bar \nu}+\delta_{-\bar \nu}\big).
\]
Then the matrix
\[
A(\mu)=\frac{1}{2} B_{G}(\nu)+\frac{1}{2}B_G(\bar \nu)
\]
has full rank. In particular the vectors \(A(\mu)\nu, A(\mu)\bar \nu\) are linearly independent. On the other hand
\[
\begin{split}
2A(\mu)\nu&=B_G(\bar \nu)\nu=G(\bar \nu) \nu-\langle d_\nu G(\bar \nu),\nu\rangle \bar \nu\\
2A(\mu)\bar \nu&=B_G( \nu)\bar \nu=G( \nu) \bar \nu-\langle d_\nu G(\nu),\bar \nu\rangle \nu
\end{split}
\]
and thus, these two vectors are linearly independent if and only if 
\begin{equation*}
G(\nu)G(\bar \nu)-\langle d_\nu G(\bar \nu),\nu\rangle \langle d_\nu G( \nu),\bar \nu\rangle \ne 0.
\end{equation*}
Since \(G\) is positive and \(\mathbb S^{n-1}\setminus \{\pm \bar \nu\}\) is connected for \(n\ge 3\), the above equation implies that 
\begin{equation}\label{e:bella}
G(\nu)G(\bar \nu)-\langle d_\nu G(\bar \nu),\nu\rangle \langle d_\nu G( \nu),\bar \nu\rangle > 0\qquad\mbox{for all \(\nu\ne \pm \bar\nu\).}
\end{equation}
Exploiting that \(G\) is even, the same can be deduced also if \(n=2\).
We now show that~\eqref{e:bella} implies~\eqref{eq:sconv} and thus the strict convexity of \(G\) (actually Step 1 of the proof shows that they are equivalent). Let \(\bar \nu\) be fixed and let us define the linear projection
\[
P_{\bar \nu} \nu=\frac{\langle d_\nu G(\bar \nu),\nu\rangle}{G(\bar \nu)} \, \bar \nu.
\]
We note that by~\eqref{eulercod1} \(P_{\bar \nu} \) is actually a projection, i.e. \(P_{\bar \nu} \circ P_{\bar \nu}=P_{\bar \nu}\). Hence,  
setting \(\nu_t=t\nu+(1-t) P_{\bar \nu} \nu\) for \(t\in [0,1]\), we have \(P_{\bar \nu} \nu_t=P_{\bar \nu} \nu\). Thus
\begin{equation}\label{train}
\nu_t-P_{\bar \nu} \nu_t=t(\nu-P_{\bar \nu} \nu).
\end{equation}
Hence, if we define \(g(t)=G(\nu_t)\), we have, for \(t\in (0,1)\),
\[
tg'(t)=t\langle d_\nu G(\nu_t), \nu-P_{\bar \nu} \nu \rangle=\langle d_\nu G(\nu_t), \nu_t-P_{\bar \nu} \nu_t \rangle >0,
\]
where in the second equality we have used equation~\eqref{train} and the last inequality follows from~\eqref{e:bella} with \(\nu=\nu_t\), and $t>0$. Hence, exploiting also  the one-homogeneity of \(G\),
\[
G(\nu)=g(1)>g(0)=G(P_{\bar \nu} \nu)=\frac{\langle d_\nu G(\bar \nu),\nu\rangle}{G(\bar \nu)} G(\bar \nu)=\langle d_\nu G(\bar \nu),\nu\rangle
\]
which proves~\eqref{eq:sconv} and concludes the proof.
\end{proof}

\appendix

\section{First variation with respect to anisotropic integrands}\label{appendixvariation}
In this section we compute the $\FF$-first variation of a varifold $V$. To this end we recall that, by identifying a \(d\)-plane \(T\) with the orthogonal projection onto \(T\), we can embed \(G(n,d)\) into \(\R^n\otimes \R^n\). Indeed we have  
\begin{equation}\label{id}
G(n,d)\approx \Big\{ T\in \R^n\otimes \R^n : T\circ T=T,\quad T^*=T,\quad \tr T=d \Big\}.
\end{equation}
With this identification, let \(T(t)\in G(n,d)\) be a smooth curve such that \(T(0)=T\). Differentiating the above equalities we get
\begin{equation}\label{utile}
T'(0)=T'(0)\circ T+T\circ T'(0),\qquad (T'(0))^*=T'(0),\qquad \tr T'(0)=0.
\end{equation}
In particular from the first equality above we obtain 
\begin{equation*}
T\circ T'(0)\circ T=0,\qquad T^\perp\circ T'(0)\circ T^\perp=0.
\end{equation*}
Hence
\[
 {\rm Tan}_T G(n,d)\subset  \big\{ S \in \R^n\otimes\R^n: S^*=S,\quad T\circ S\circ T=0,\quad T^\perp\circ S\circ T^\perp=0\big\}.
 \]
 Since \({\rm dim }  {\rm Tan}_T\, G(n,d)={\rm dim}\, G(n,d)=d(n-d)\) the above inclusion is actually an equality.   To compute the anisotropic first variation of a varifold we need the following simple Lemma:
 \begin{lemma}\label{der}
Let  \(T\in G(n,d)\) and  \(L\in \R^n\otimes \R^n\), and let us define \(T(t)\in G(n,d)\)  as the orthogonal projection onto  \((\Id+tL)(T)\) (recall the identification~\eqref{id}). Then 
\begin{equation*}\label{utilissima}
T'(0)=T^\perp \circ L\circ T+(T^\perp \circ L\circ T)^* \in  {\rm Tan}_T G(n,d).
\end{equation*}
 \end{lemma}
 
 \begin{proof}
One easily checks that \(T(t)\) is  a smooth function of \(T\) for \(t\) small. Since
\[
T(t)\circ (\Id+tL)\circ T=(\Id+tL)\circ T,
\]
 differentiating we get
\begin{equation}\label{percheno}
T'(0)\circ T =(\Id-T)\circ L\circ T=T^\perp \circ L\circ T.
\end{equation}
Using that \((T'(0))^*=T'(0)\), \(T^*=T\),  the first equation in~\eqref{utile} and \eqref{percheno}, 
one obtains
\begin{equation*}
\begin{split}
T'(0)&=T'(0)\circ T+T\circ T'(0)\\
&=T'(0)\circ T+(T'(0)\circ T)^*=T^\perp \circ L\circ T+(T^\perp \circ L\circ T)^*,
\end{split}
\end{equation*}
and this concludes the proof.
\end{proof}
We are now ready to compute the first variation of an anisotropic energy:

\begin{lemma} Let \(F\in C^1(\Omega\times G(n,d))\) and \(V\in \mathbb V_d(\Omega)\), then for \(g\in C^1_c(\Omega,\R^n)\) we have
 \begin{equation}\label{varapp}
\delta_F V(g)
=\int_{G(\Omega)} \Big[\langle d_xF(x,T),g(x)\rangle+ B_F(x,T):Dg(x)  \Big] dV(x,T),
 \end{equation}
 where the matrix \(B_F(x,T)\in \R^n\otimes\R^n\) is uniquely defined by 
\begin{equation}\label{Bapp}
B_F(x,T): L\, := F(x,T) (T:L)+\big\langle d_T F(x,T) ,\,T^\perp\circ L\circ T +(T^\perp\circ L\circ T)^*\big\rangle
 \end{equation} 
 for all \(L\in \R^n\otimes \R^n\).
\end{lemma}

\begin{proof}
For \(g\in C_c^1(\Omega,\R^n)\) let \(\varphi_t(x)=x+tg(x)\) which is a diffeomorphism of \(\Omega\) into itself for \(t\ll1\).  We have 
\begin{equation*}\label{derivate}
\begin{split}
\delta_F V(g)&=\frac{d}{dt}\FF (\varphi_t^{\#}V) \Big|_{t=0}\\ 
&=\frac{d}{dt} \int_{G(\Omega)}  F (\varphi_t(x), d\varphi_t (T)) J\varphi_t(x,T) dV(x,T)\Big|_{t=0}\\
&=\int_{G(\Omega)}  \frac{d}{dt}F (\varphi_t(x), T)  dV(x,T)\Big|_{t=0}+\int_{G(\Omega)}  \frac{d}{dt} F (x, d\varphi_t (T)) dV(x,T)\Big|_{t=0}\\
&\quad+\int_{G(\Omega)}  F (x, T) \frac{d}{dt}J\varphi_t(x,T) \Big|_{t=0}dV(x,T).
\end{split}
\end{equation*}
Equation~\eqref{varapp} then follows by the definition of \(B_F(x,T)\),~\eqref{Bapp}, and  the equalities
\begin{align}
 \frac{d}{dt}F (\varphi_t(x), T)\Big|_{t=0}&=\langle d_xF(x,T),g(x)\rangle, \label{11}\\
  \frac{d}{dt}J\varphi_t(x,T)\Big|_{t=0}&=T:Dg(x),\label{33}\\
 \frac{d}{dt}F (x, d\varphi_t (T))\Big|_{t=0} &=\big\langle d_T F(x,T) ,\,T^\perp\circ Dg(x)\circ T +(T^\perp\circ Dg(x)\circ T)^*\big\rangle.\label{22}
\end{align}
Here~\eqref{11} is trivial,~\eqref{33} is a classical computation, see for instance~\cite[Section 2.5]{SimonLN}, and~\eqref{22} follows from Lemma~\ref{der}.
\end{proof}

\section{Proof of Lemma~\ref{rect}}\label{Lemma}
In this Section we prove  Lemma~\ref{rect}. Let us start recalling 
 the following rectifiability criterion due to Preiss, see~\cite[Theorem 5.3]{Preiss}.

\begin{theorem}\label{Preiss}
Let $\mu$ be a measure on $\R^n$ and assume that at 
$\mu$-a.e. $x$ the following two conditions are satisfied:
\begin{itemize}
\item[(I)]If we set $\alpha=\alpha_d=1-2^{-d-6}$ and 
$$
E_r(x):= \left \{z \in B_r(x):\, 
\textrm{ 
\(\exists s \in (0,r) \) satisfying  \(\dfrac{\mu(B_s(z))}{\omega_ d s^d}\leq \alpha  \frac{\mu(B_r(x))}{\omega_d r^d}\)}\right \},
$$
then
$$\liminf_{r\to 0}\frac{\mu(E_r(x))}{\mu(B_r(x))}=0;$$
\item[(II)]  If we set $\beta=\beta_d=2^{-d-9}d^{-4}$ and 
$$
F_r(x):=\sup_{T\in G(n,d)}\left \{\inf_{z\in (x+T)\cap B_r(x)} \frac{\mu(B_{\beta r}(z))}{\mu(B_r(x))}\right\},
$$
then
$$ \liminf_{r\to 0}F_r(x)>0.$$
\end{itemize}
Then $\mu$ is a $d$-rectifiable measure.
\end{theorem}

\begin{proof}[Proof of Lemma~\ref{rect}]
By replacing \(\mu\) with \(\mu\res \Omega'\), where \(\Omega'\cc \Omega\), we can assume that \(\mu\) is defined on the whole \(\R^n\).
We are going to prove that \(\mu\) verifies conditions (I) and (II) in Theorem~\ref{Preiss}. 

Let us start by verifying condition (I). 
Given $\e,m>0$, let
$$E(\e,m):=\left \{z\in\R^n: \quad \frac{\mu(B_r(z))}{\omega_d r^d} > m\, \textrm{ for all  \(r\in (0,\e)\)} \right \},$$
and, for \(\alpha=\alpha_d\) as in Theorem~\ref{Preiss} and \(\gamma\in (1,1/\alpha)\), set
$$
\widehat{E}(\e,m):=E(\e,\alpha \gamma m)\setminus  \bigcup_{k=1}^{\infty}E\left(\frac \e k, m\right).
$$
If \(x\) is  such that  \(0<\Theta^d_*(x,\mu)<+\infty\), then 
\(x\in \widehat{E}(\bar \e,\bar m)\) for some positive \(\bar \e\)  and \( \bar m\) 
such that \( \alpha\gamma  \bar m <\Theta^d_*(x,\mu)< \bar m\), hence
\[
\{ 0<\Theta^d_*(x,\mu)<+\infty\} \subset \bigcup_{m>0} \bigcup_{\e>0} \widehat{E}(\e,m).
\]
Let  now \(x\in \widehat{E}(\e,m)\) be a density point for \(\widehat{E}(\e,m)\):
\begin{equation}\label{echecazzo}
\lim_{r\to 0} \frac{\mu (B_r(x)\setminus \widehat{E}(\e,m))}{\mu(B_r(x))}=0.
\end{equation}
Note that \(x\in \widehat{E}(\e,m)\) implies that \( \alpha\gamma m\le\Theta^d_*(x,\mu)\le  m<\gamma m\). Hence,  if $(r_k)_k$ is a sequence verifying \(r_k\to 0\), \(r_k< \e\) and such 
that \(\Theta^d_*(x,\mu)=\lim_{k} \mu(B_{r_k}(x))/\omega_d r_k^d\), then, for \(k\) large enough,  
\[
E_{r_k}(x) \subset  B_{r_k}(x)\setminus E(\e,\alpha \gamma m) 
\subset  B_{r_k}(x)\setminus  \widehat E(\e,m),
\]
which, together with \eqref{echecazzo}, proves that \(\mu\) verifies condition (I).

We now verify condition (II). Let \(x\) be a point such that all the tangent measures at \(x\) are translation invariant in the directions of \(T_x\) and such that \( 0<\Theta^d_*(x,\mu)\le \Theta^{d*}(x,\mu)<+\infty\). Note  that the latter condition implies that for every \(\sigma\in {\rm Tan}(x,\mu)\)
\begin{equation*}
\frac{\Theta^d_*(x,\mu)}{\Theta^{d*}(x,\mu)} t^d\le \sigma( B_t)\le  \frac{\Theta^{d*}(x,\mu)}{\Theta^{d}_*(x,\mu)} t^d\qquad\textrm{for all \(t\in (0,1)\)}.
\end{equation*}
In particular, \(0\in \spt\, \sigma\) for all \(\sigma\in {\rm Tan}(x,\mu)\). 
Let us choose a sequence $r_i\to 0$ and \(z_{r_i}\in (x+T_x)\cap B_{r_i}(x)\), such that
\[
\begin{split}
\liminf_{r\to 0} \Bigg\{\inf_{z\in (x+T_x)\cap B_r(x)}\frac{\mu(B_{\beta r}(z))}{\mu(B_r(x))}\Bigg\}&=\lim_{i\to \infty} \frac{\mu (B_{\beta r_i}(z_{r_i}))}{\mu(B_{r_i}(x))}\\
&\ge \lim_{i\to \infty} \mu_{x,r_i}\left (B_{\beta}\left (\frac{z_{r_i}-x}{r_i}\right )\right ),
\end{split}
\]
where \(\mu_{x,r_i}\) is defined in \eqref{blowup} and \(\beta=\beta_d\) is as in Theorem~\ref{Preiss}. Up to subsequences we have that 
$$
\lim_{i\to \infty} \mu_{x,r_i}\weakstarto  \sigma\in {\rm Tan}(x,\mu)\qquad\textrm{and}\qquad
\lim_{i\to \infty} \frac{z_{r_i}-x}{r_i}=z \in \bar{B}\cap T_x.
$$
Hence 
\begin{equation*}\label{2}
\liminf_{r\to 0} \Bigg\{\inf_{z_r\in (x+T_x)\cap B_r(x)}\frac{\mu(B_{\beta r}(z_r))}{\mu(B_r(x))}\Bigg\}\geq \sigma (B_{\beta }(z)).
 \end{equation*}
Let  $z'\in B_{\beta/2}(z)\cap T_x $ such that $B_{\beta/2}(z')\subset B_{{\beta}}(z)\cap B$. Since  $\sigma$ is translation invariant in the directions of  $T_x$ 
\begin{equation*}\label{1}
\sigma(B_{\beta }(z))
\geq \sigma (B_{\frac{\beta}{2}}(z'))=\sigma (B_{\frac{\beta}{2}}(0))>0,
\end{equation*}
where in the last inequality we have used that \(0\in \spt\, \sigma\). Thus
\[
\liminf_{r\to 0} F_r(x)\ge \liminf_{r\to 0} \Bigg\{\inf_{z\in (x+T_x)\cap B_r(x)}\frac{\mu(B_{\beta r}(z))}{\mu(B_r(x))}\Bigg\}>0,
\]
implying that also condition  (II) in Theorem~\ref{Preiss} is satisfied. Hence  \(\mu\) is \(d\)-rectifiable. 
 In particular for \(\mu\)-a.e. \(x\), \({\rm Tan}(x,\mu)=\{\omega_d^{-1}\H^d\res( T_x K\cap B)\}\). Since, by assumption, \(\mu\) is invariant along the directions of \(T_x\),  this implies that 
 \(T_x=T_x K\) and concludes the proof.
\end{proof}


\end{document}